\documentclass[reqno,11pt]{amsart}
\usepackage{a4wide,color,eucal,enumerate,mathrsfs}
\usepackage[normalem]{ulem}
\usepackage{amsmath,amssymb,epsfig,amsthm} 
\usepackage[latin1]{inputenc}
\usepackage{psfrag}

\def\rr{{\mathbb R}}

\def\fz{\infty}

\def\dist{{\mathop\mathrm{\,dist\,}}}
\def\loc{{\mathop\mathrm{\,loc\,}}}

\def\lz{\lambda}
\def\dz{\delta}

\def\ez{\epsilon}

\def\bz{\beta}

\def\gz{{\gamma}}

\def\bint{{\ifinner\rlap{\bf\kern.25em--}
\int\else\rlap{\bf\kern.45em--}\int\fi}\ignorespaces}

\def\bbint{{\ifinner\rlap{\bf\kern.35em--}
\hspace{0.078cm}\int\else\rlap{\bf\kern.45em--}\int\fi}\ignorespaces}

\def\diam{{\mathop\mathrm{\,diam\,}}}

\def\subsetneq{{\hspace{0.2cm}\stackrel
\subset{\scriptstyle\ne}\hspace{0.15cm}}}

\newtheorem{thm}{Theorem}[section]
\newtheorem{prop}[thm]{Proposition}
\numberwithin{equation}{section}

\theoremstyle{remark}

\def\bint{{\ifinner\rlap{\bf\kern.35em--}
\int\else\rlap{\bf\kern.45em--}\int\fi}\ignorespaces}

\title{$C^{1}$-regularity of planar $\infty$-harmonic functions---Revisit }
\author{ Yi Ru-Ya Zhang and Yuan Zhou}
\address{Y. Zhang: Hausdorff Center for Mathematics, Endenicher Allee 62, Bonn 53115, Germany}
\email{yizhang@math.uni-bonn.de}
\address{Y. Zhou: Department of Mathematics, Beihang University, Beijing 100191, P.R. China}
\email{yuanzhou@buaa.edu.cn}
\date{\today}

\begin{document}

 \allowdisplaybreaks
\arraycolsep=1pt
\maketitle
\begin{center}
\begin{minipage}{13cm}
{\bf Abstract.}
   In the seminal paper  [Arch. Ration. Mech. Anal. 176 (2005), 351--361],
  Savin  proved the $C^1$-regularity of planar $\infty$-harmonic functions $u$. Here we give a new understanding of it from a capacity viewpoint and drop several high  technique arguments therein. Our argument  is essentially based on a topological lemma of Savin, a flat estimate by Evans and Smart, 
  $W^{1,2}_\loc$-regularity   of  $|Du|$ and Crandall's  flow for infinity harmonic functions.
\end{minipage}
\end{center}

\section{Introduction}

Let $n\ge2$ and $\Omega\subset \mathbb R^n$ be a domain (an open connected subset).
 A  function $u\in C^0(\Omega)$  is  called $\infty$-harmonic   in $\Omega$ if
\begin{equation}\label{infty equ}
-\Delta_{\infty} u := -\sum_{i,j=1}^nu_{x_i}u_{x_j}u_{x_ix_j}   =0 \quad {\rm in}\ \Omega
\end{equation}
in   viscosity sense; see \cite{j1993}. The existence and uniqueness of  $\infty$-harmonic functions has been established by Jessen in \cite{j1993}. Their regularity is  the main issue in this field.

When $n=2$,  based on the planar topology, the linear approximation property by Crandall-Evans \cite{ce},
 and the comparison property with cones by   Crandall-Evans-Gareipy \cite{ceg},
 in the seminal paper \cite{s05} Savin proved that
 \begin{thm}\label{sa} If
  $u $ is an $\infty$-harmonic function in a domain $\Omega\subset\rr^2$,
    then $u\in C^1(\Omega)$.
\end{thm}

However Savin's  approach  heavily depends on the planar topology, which makes it difficult to generalize to the higher dimension.

On the other hand,  via specific PDE approach (and hence completely different from  Savin' approach),
for any $n\ge2$ Evans-Smart \cite{es11a,es11b} established the everywhere differentiability of $\infty$-harmonic functions $u$ in $\mathbb R^n$.
Indeed, they approximated $u$ in $C^0$ via $e^{\frac1{2\ez}|p|^2}$-harmonic functions  $u^\ez$,
 and built up certain flatness estimate for $u^\ez$; see Lemma~\ref{keyp1} for a version of it. From this
 and  the linear approximation property they resulted the  everywhere differentiability of $u$.
Recently in the plane, Koch-Zhang-Zhou \cite{kzz} further obtained
a quantative $W^{1,2}_\loc$-regularity of $|Du|$
by building up a structural identity for $e^{\frac1{2\ez}|p|^2}$-harmonic equation, and then showing uniform $W^{1,2}_\loc$-regularity of $|Du^\ez|$ and the  Sobolev  convergence of  $u^\ez\to u$.

In this paper, we give a new viewpoint of  Savin's  $C^1$-regularity proof via a capacity argument.  This allows us to skip certain high technique arguments  in his original proof.  The key point is to show the continuity of $|Du|$ when $|Du|\neq 0$ via the
 $W^{1,2}_\loc$-regularity of $|Du|$ by Koch-Zhang-Zhou \cite{kzz}, the existence of a curve with large $|\nabla u|$ by Crandall \cite{c08} and the
  existence of  a continuum with small  $|\nabla u|$ in the original paper of Savin \cite{s05}; indeed one directly concludes the logarithmic moduli of continuity of $|Du|$ when $|Du|\neq 0$ with this method. Then combining with  a flatness estimate of $u$ by
Evans-Smart \cite{es11a}, we obtain the continuity of $Du$ when  $|Du|\neq 0$. The continuity of $Du$ at $\{Du=0\}$ is a direct consequence of the upper semi-continuity of $|Du|$ at differentiable points of $u$ \cite{ceg} and  the everywhere differentiability of $\infty$-harmonic functions \cite{es11a,es11b}.

We end the introduction by recalling the following conjecture; see \cite{kzz} for details.

 \noindent{\bf Conjecture.} Let $u $ be an $\infty$-harmonic function in a domain $\Omega\subset\rr^2$.
Then $|Du|\in W^{1,p}_\loc(\Omega)$ for some $p>2$.

If this conjecture were true, then one would directly conclude the continuity of $|Du|$. This together with the flat estimate of Evans-Smart would also imply Proposition~\ref{keyp}, and then  the continuity of $Du$.

\section{Proof of Theorem 1.1}

Recall that  by Evans-Smart \cite{es11a,es11b},
   a  planar $\fz$-harmonic function $u$ is differentiable everywhere and every point is a Lebesgue point of $Du$.
 Theorem 1.1 then is a direct consequence of the following result.
 \begin{prop}\label{keyp} Assume that $u$ is  a planar $\fz$-harmonic function in $B(0,4)$  and satisfies $u(0)=0$ and
 $Du(0)=e_2$.   If  \begin{equation}\label{flatass}\sup_{B(0,4)}|u(x) -x_2|\le \lz \end{equation}
 for some $\lz\in(0,1)$, then $$\limsup_{x\to 0}|Du(x)-e_2|\le 1-C\lz^{1/2},$$
 where $C\ge1$ is  an absolute constant.
 \end{prop}

For reader's convenience we give the details of Theorem 1.1 via Proposition \ref{keyp} as below.

\begin{proof}[Proof of Theorem 1.1]
We show that $Du$ is continuous at any given point $\bar x\in \Omega$.
 For simplicity, we may assume that $\bar x=0$.
 If $Du(0)=0$, by the upper-semicontinuity of  $|Du|$ (see \cite{ceg}) we immediately obtain the continuity of $Du$  at $0$.
 Assume  that $Du(0)\ne 0$. Up to some suitable scaling and rotation, we may assume that $Du(0)=e_2$.
For any $\lz>0$, by the differentiability of $u$ at $0$, there exists an
 $r_\lz\in(0,\frac12\dist(x,\partial\Omega))$ such that
  $$\sup_{x\in B(0, r_\lz)}\frac1 {   r_\lz}|u(x)-x_2|\le \lz .$$
 Let  $v(x)=u( 4r_\lz x)/ 4r_\lz$ in $B(0,4)$. Then $v$ is $\fz$-harmonic in $B(0,4)$, $v(0)=0$, $Dv(0)=e_2$ and
    $$\sup_{x\in B(0, 4)} |v(x)-x_2|\le 4\lz .$$
    Applying Proposition \ref{keyp} and $Du(r_\lz x)=Dv(x)$ we have
   $$ \lim_{x\to0} |Du(x)-e_2|=\lim_{x\to0} |Dv(x)-e_2|\le  C\lz^{1/2}.$$
   By the arbitrariness of $\lz>0$ we conclude   $  \lim_{x\to0}Du(x)=e_2  $ as desired.
\end{proof}

Below we
  prove Proposition \ref{keyp} with the aid of Propositions \ref{keyp1} and \ref{keyp2}.
Proposition  \ref{keyp1}
is a consequence of the flatness estimate by Evans-Smart \cite{es11a}; some details are given for reader's convenience.
 \begin{prop}\label{keyp1}
 Let $u$ be as in Proposition \ref{keyp}. 
 If  \eqref{flatass} holds for some $\lz\in(0,1)$,
 then $$|Du(x)|^2\le u_{x_2}+C_0\lz\quad\mbox{in $B(0,1)$},$$
 where $C_0\ge 1$ is an absolute constant.
 \end{prop}

 \begin{proof}
 By  \cite{e03,ey04} and \cite{es11b}
for $\epsilon\in(0,1]$  there exists a unique  solution $u^\ez\in C^\infty(B(0,3))\cap C(\overline {B(0,3)})$ to
\begin{equation*} \label{regular infty equ}
 -\Delta_{\infty} u^\ez -\ez \Delta u^\ez=0 \ \text{in $B(0,3)$}, \quad u^\ez=v   \ \text{on $\partial B(0,3)$}
\end{equation*}
so that $u^{\ez}\to u$ in $C^0(\overline{B(0,3)})$   as $\ez\to0$, and
$$\max_{ \overline{V}} |Du^{\ez}|\le C  \dist(V,\,\partial B(0,3)) \quad\forall V\Subset B(0,3).$$
If $\ez>0$ is small enough, we have
 $$\sup_{B(0,3)}|u^\ez (x)-u^\ez(0)-x_2 |\le 2 \lz .$$
By \cite{es11b} and also \cite{es11a,swz}, we further obtain
 $$|Du^\ez|^2\le u^\ez_{x_2}+C\lz^{1/2} \quad\mbox{in $B(0,2)$}. $$
  This implies that
 $$\bint_{B(\bar x,r)}|Du^\ez|^2\,dx \le \bint_{B(\bar x,r)} u^\ez_{x_2}\,dx +C\lz^{1/2} \quad\mbox{$\bar x\in B(0,1)$ and  $r\in(0,1)$}.$$
 Letting $\ez\to0$  and noting $|Du^\ez|\to |Du|$ weakly in $L^2(B(0,1))$ (indeed we even have strong convergence here \cite{kzz}), we conclude that
 $$\bint_{B(\bar x,r)}|Du |^2\,dx \le \bint_{B(\bar x,r)} u _{x_2}\,dx +C\lz^{1/2}\quad\mbox{$\bar x\in B(0,1)$ and  $r\in(0,1)$}.$$
Sending $r\to0$ and recalling that $\bar x$ is a Lebesgue point of $Du$ as given in \cite{es11b},   we eventually get
 $$|Du |^2\le u _{x_2}+C\lz^{1/2} \quad\mbox{in $B(0,1)$}. $$
 \end{proof}

Proposition \ref{keyp2} was proved by  Savin \cite{s05} via a topological argument.
For the convenience of the reader,  we sketch the proof.
  \begin{prop}\label{keyp2}  Let $u$ be as in Proposition \ref{keyp}. 
    If \eqref{flatass} holds for some $\lz\in(0,1)$,
and
  $$\mu:=\liminf_{x\to0} |Du(x)|<1,$$
  then for sufficiently small $r>0$, there is a continuum $\eta$ joining  $ \partial B(0,1)$ and $ \partial B(0,r)  $
  so that
 $$|Du  |<  \mu+9\lz\quad \mbox{in  $\eta$.}$$

 \end{prop}

 \begin{proof}  The proof combines some argument  from
 \cite{s05,wy,fwz}.
  Without loss of generality, we may assume that $\mu \le 1-7\lz$.
 Indeed, by comparison property with cones in \cite{ceg}, (2.1)
implies that $|Du|\le 1+2\lz$ in $B(0,3)$. If $\mu>1-7\lz$, this implies that
$|Du|< \mu+9\lz$ in $B(0,3)$, and hence any line segment joining
$\partial B(0,r) $  and $\partial B(0,1)$ gives a desired continuum $\eta.$

 Below let  $r >0$ be sufficiently small such that $\inf_{x\in B(0,r )}|Du(x)|>\mu-\lz$.
 By the upper semicontinuity of $|Du|$, the set
  $U:= \{x\in B(0,4): |Du(x)|< \mu+\lz \}$ is   open and nonempty, and moreover,
  $|Du|\ge \mu+\lz$ on $\partial U\cap B(0,4)$, where we note that 
 $|Du(0)|=1> \mu+\lz$ implies that $U\subsetneq B(0,4)$.  
 By 
  $ \mu\le 1-7\lz$, there must be a connected component $U_0$ of $U$ such that
  $U_0\cap B(0,r/8)\ne\emptyset $.
  Denote by $U_1$   a connected component  of $U_0\cap B(0,r)$  satisfying  $U_1 \cap B(0,r/8)\ne\emptyset $.
    Then $u$ is not a linear function in $U_1$.
  Otherwise, $u=a\cdot x+b$  in $U_1$, and hence  in $\overline {U_1 }$, for some vector $a$ with $|a|<\mu+\lz$.
  Given any $x\in U_1\cap B(0,r/8)$, there exists a point $w\in \partial U_1\cap B(0,r/2)\subset \partial U\cap B(0,4)$ such that
 $|x-w|= \dist(x,\partial U_1)$.  Then $|Du(w)|\ge \mu+\lz$. On the  other hand,
  for any  unit vector $e\not\perp x-w$,  we can find a $h\ne0$ so that the line segment
  $(w,w+he)\subset B(x,\dist(x,\partial U_1))\subset U_1$, and hence, by $u=a\cdot x+b$   in $\overline {U_1 }$ one concludes that 
  $$Du(w)\cdot e=\lim_{\dz\to0+} \frac{u(w+\dz he)-u(w)}{\dz h}=a\cdot e.$$
 This gives that  $ Du(w) = a $  and  $|Du(w)|=|a|<\mu+\lz$, which
 is a contradiction.

Since  $u$ is not linear in the connected open set $U_1$,
  there exists a line segment $[\bar x ,\bar y]\subset U_1$, a point $\bar z\in (\bar x ,\bar y)$
and a linear function $l(x)=a_0\cdot x+b_0 $ in $\rr^2$
with $a_0=[u(\bar y)-u(\bar x )]\frac{\bar y-\bar x }{|\bar y-\bar x |^2}$
such that either
 \begin{equation}\label{xcase1}
u\ge l\quad {\rm on}\ [\bar x ,\bar y],\quad u(\bar x )>l(\bar x ),\quad u(\bar z)=l(\bar z),\quad
u(\bar y)>l(\bar y);
\end{equation}
or
 \begin{equation*}
u\le l\quad {\rm on}\ [\bar x ,\bar y],\quad u(\bar x )<l(\bar x ),\quad u(\bar z)=l(\bar z),\quad
u(\bar y)<l(\bar y).
\end{equation*}
Up to considering $-u$,
  we may assume that \eqref{xcase1} holds.
Since $u-l$ reaches it minimal in $[\bar x ,\bar y]$ at $\bar z$, we have
  $$\left(a_0-Du(\bar z)\right) \cdot \left( z -\bar z \right)=0 \quad\forall z\in [\bar x ,\bar y],$$
 which, together with \eqref{xcase1}, yields  that
  $$\bar x ,\bar y\in W:=\{y\in B(0,4): u(y)>u(\bar z)+Du(\bar z) \cdot (y-\bar z)\}.$$
  Denote by $W_{\bar x}$ (resp. $W_{\bar y}$) the connected component  of $W$ which contains  $\bar x$ (resp. $\bar y$).
  Note  that $\bar z\in U_1\subset U\cap B(0,r)$ implies that $$\mu-\lz<|Du(\bar z)|<\mu+\lz.$$
 The proof is then divided into 2 steps.

  \medskip

 \noindent {\it Step 1.}  Via planar topology, we prove   $W_{\bar x}\ne W_{\bar y}$   by contradiction.

 Suppose that $W_{\bar x}=W_{\bar y}$.
Then there exists a simple curve $\gamma_0\subset W$ joining $\bar x $ to $\bar y$.
Let $\gz=\gamma_0\cup
 [\bar x ,\bar y]\subset B(0,4)$, which is a simple closed curve, and $V$ be the open set bounded by $\gz$ so that $\gz=\partial V$.
 Without loss of generality, we may assume that
there exists a small $\beta >0$ such that
 $$ B(\bar z,\beta)\cap
 \{y\in\rr^2:0<\measuredangle (y-\bar z,\bar y-\bar x )<\pi\}\subset V.$$
 Let $ \nu$ be a  unit  vector so  that $\nu \cdot (\bar y-\bar x )=0$
  and $\bar z+\frac12\bz\nu \in V$.
From  the compactness of $\gz_0\subset W$, it follows that
 $$\inf_{y\in\gz}[u(y)-u(\bar z)-Du(\bar z)\cdot (y-\bar z)]>0,$$
by which, there is  a small
$\ez_0>0$  such that
 $$u(y)\ge u(\bar z)+(Du(\bar z)+\ez_0 \nu)\cdot (y-\bar z),\quad
\, \forall y\in \gz_0.$$
Since $\nu\cdot(\bar x -\bar y)=0$, by (2.2) one also has that
$$u(y)\ge l(y)=u(\bar z)+(Du(\bar z)+\ez_0 \nu)\cdot (y-\bar z),\quad \, \forall y\in [\bar x ,\bar y].$$
The comparison principle in \cite{j1993} then gives
$$u(y)\ge u(\bar z)+(Du(\bar z)+\ez_0 \nu )\cdot (y-\bar z), \quad\, \forall y\in V $$
and hence
 $$Du(\bar z)\cdot \nu=\lim_{s\to0^+}\frac{u(\bar z+s\nu)- u(\bar z)}s \ge  Du(\bar z)\cdot  \nu +\ez_0,$$
which is contradiction.

\medskip
\noindent  {\it Step 2.}   Construct a desired continuum $\eta$.

 By \eqref{flatass}, a direct calculation yields that
 $$ \{y\in\rr^2: (e_2-Du(\bar z))\cdot (y-\bar z)<-2\lz\}\cap B(0,4)
\subset B(0,4)\setminus W$$
and $$  \{y\in\rr^2: (e_2-Du(\bar z))\cdot (y-\bar z)> 2\lz\} \cap B(0,4)
\subset W.$$
By Step 1, $W$ contains at least two  distinct connected components  $W_{\bar x}$ and $W_{\bar y}$. Thus,
at least
one of  them (say $W_{\bar x}$)  is contained in  $$ \mathscr{S}:= \{y\in\rr^2: |(e_2-Du(\bar z))\cdot (y-\bar z)|\le
2\lz\}\cap B(0,4).$$
Note that $\overline {W_{\bar x}}\not\subset B(0,4)$. Indeed, otherwise,  we have
 $u(y)=u(\bar z)+ Du(\bar z)\cdot (y-\bar z)$
on $\partial W_{\bar x}$
 and hence in $W_{\bar x}$ by the comparison principle \cite{j1993},
which  contradicts to the definition of $W_{\bar x}$.

We claim that \begin{equation}\label{claim}
\mbox{$|Du(z)|\le  |Du(\bar z)|+6\lz$ for all $z\in W_{\bar x}\cap B(0,3)$}.\end{equation}
Assume this claim holds for the moment.
Since $|Du(\bar z)|\le\mu+ \lz$, this claim gives
$$\mbox{$|Du(z)|\le  \mu+7\lz$ for all $z\in W_{\bar x}\cap B(0,3)$}.$$
Then  any curve in $ W_{\bar x}  \cap B(0,3)
$ joining $\bar x\in B(0,r)$ and $\partial B(0,1)$ gives a desired continuum $\eta$.

Finally, we prove above claim \eqref{claim} as below.
It suffices to prove that
given any $z\in W_{\bar x}\cap B(0,3)$, one has
\begin{equation}\label{tt} u(y)\le  u( z)+ (|Du(\bar z)|+6\lz)|y-  z|\quad\mbox{for all $y\in \partial(W_{\bar x}\cap
B(0,4))$}.
\end{equation} Indeed, by
   the comparison property with cones in \cite{ceg}, \eqref{tt}  implies that 
  $$  u(y)\le  u( z)+ (|Du(\bar z)|+6\lz)|y-  z|\quad\mbox{for all $y\in  W_{\bar x}\cap
B(0,4) $},$$
which gives that $|Du(z)|\le Du(\bar z)+6\lz$ as desired.

To see \eqref{tt}, note that $\partial(W_{\bar x}\cap
B(0,4))\subset [\partial W_{\bar x}\cap B(0,4)]\cup [\overline{W_{\bar x}}\cap
 \partial {B(0,4)} ]$. Let $z\in W_{\bar x}\cap B(0,3)$. One always has
$$u(y)=u(\bar z)+ Du(\bar z)\cdot (y-\bar z)\le u(z)+Du(\bar z)\cdot (y-z)\le u(z)+|Du(\bar z)||y-z|  \quad \forall\, y\in \partial W_{\bar x}\cap B(0,4).$$
 On the other hand,
  for every $y\in \overline{W_{\bar x}}\cap
 \partial {B(0,4)} ,$   by \eqref{flatass} 
 one   has
$$ u(y) \le   u(   z)+ e_2\cdot (y-   z) +2\lz
 \le u(  z)+ Du( \bar z)\cdot (y-   z) +
|(e_2-Du(\bar z))\cdot (y- z)|+ 2\lz.$$
Thanks to  $y,z\in\overline{W_{\bar x}}\subset \mathscr S$, this leads to 
$$|(e_2-Du(\bar z))\cdot (y- z)|\le  |(e_2-Du(\bar z))\cdot (y-\bar  z)|+ |(e_2-Du(\bar z))\cdot (\bar z- z)|\le 4\lz.$$
Using   $|y-z|\ge1$, we finally have
 \begin{align*} u(y)
\le u(  z)+ Du(\bar z)\cdot (y-   z)+6\lz
 \le u(  z)+ (|Du(\bar z)|+6\lz)|y-  z|
 \end{align*}
as desired.
 \end{proof}

\begin{proof}[Proof of Proportion \ref{keyp}]
By the comparison property with  cones in \cite{ceg},
one has $|Du(x)|\le 1+2 \lz$ in $B(0,3)$ and hence $\mu\le 1+2 \lz$; see \cite{ceg}.
We claim that $\mu + 9\lz\ge 1$.
Note that if the claim is true, then by Proportion \ref{keyp1}, for $x$ close to $0$, we have
  $$(1-10 \lz )^2\le |Du(x)|^2\le u_{x_2}+C_0\lz^{1/2}\le |Du(x)|+C_0\lz^{1/2}\le 1+C \lz^{1/2}$$
  and hence,
  $$ \mbox{$ | u_{x_2}(x)-1|\le   C\lz^{1/2}  $
 and $|u_{x_1}(x)|\le |Du(x)|-u_{x_2}(x)\le 1+C\lz-(1-10\lz )^2\le C\lz$ }.$$
 These allow us to conclude  $|Du(x)-e_2 |\le C\lz^{1/2}$ as desired.

We prove the above claim  by contradiction.
Assume that  $\mu + 9\lz<1$.
For sufficiently small  $r>0$,   by Proposition \ref{keyp2},
  there exists  a continuum $\eta\subset  \overline {B(0,1)\setminus B(0,r)} $
  joining $\partial B(0,r)$ and $\partial B(0,1 )$ so that
   $|Du  |<  \mu+9\lz$  in $\eta$.
Recall also that  Crandall \cite{c08} built up a Lipschitz  curve $\xi\subset \overline {B(0,1)}$   joining $0$ and
 $ \partial B(0,1)$  so that
 $    |Du |\ge|Du(0)|=1$ in $\xi$.
 Since $\dist(\xi,\eta)<r$ and $\min\{\diam\xi,\diam\eta\}\ge1/2$,  one has
 $$C_1\ln\frac1r\le {\rm
 Cap}\,(\xi,\eta,B(0,2)):=\inf\{\|\nabla w\|^2_{L^2(B(0,2))}: w\in C^0(B(0,2)), w\ge 1\ \mbox{in $\xi$}, w\le0\ \mbox{in $\eta$}\};$$
 see e.g.\ \cite{hk}.
By Koch-Zhang-Zhou \cite{kzz},     $|Du|\in W^{1,2}_\loc(B(0,4))$ and
    $$\int_{ B(0,2)}|D|Du| |^2\,dx\le C\bint_{B(0, 3)}|Du|^2\,dx \le C_2. $$
Write
 $$w=\frac1{1-\mu-9\lz}[|\nabla u|-( \mu+9\lz)].$$
    Up to a continuous approximation, one has
  $$C_1\ln \frac1r\le\int_{ B(0,2)}|Dw |^2\,dx=  \frac1{(1-\mu-9\lz)^2}\int_{B(0, 2)}|D|Du||^2\,dx \le \frac{C_2}{(1-\mu-9\lz)^2}.$$
When $r\to0$, one has
  $$(1-\mu-9\lz)^2\le C\left(\ln\frac1r\right)^{-1}\to0 ,$$
  that is, $ \mu+9\lz=1$,   which is a contradiction.
\end{proof}

%

\end{document}